\documentclass[a4paper,10pt]{amsart}
\usepackage[utf8]{inputenc}
\usepackage{mathtools}
\usepackage{amssymb}
\usepackage{amsthm}
\usepackage{xcolor}
\usepackage{tikz}
\usepackage{pgfplots}
\pgfplotsset{compat=1.18}
\usepackage{fullpage}
\usepackage{todonotes}
\usepackage{orcidlink}

\DeclareSymbolFont{bbold}{U}{bbold}{m}{n}
\DeclareSymbolFontAlphabet{\mathbbold}{bbold}

\newcommand{\R}{\mathbb{R}}
\newcommand{\C}{\mathbb{C}}
\newcommand{\K}{\mathbb{K}}
\newcommand{\1}{\mathbbold{1}}
\newcommand{\from}{\colon}
\DeclareMathOperator{\lin}{span}
\DeclareMathOperator{\sech}{sech}
\newcommand{\calH}{\mathcal{H}}
\newcommand{\calG}{\mathcal{G}}
\newcommand{\calL}{\mathcal{L}}
\newcommand{\calK}{\mathcal{K}}

\makeatletter

\newcommand\MyPairedDelimiter{%
  \@ifstar{\My@Paired@Delimiter{{}}}
          {\My@Paired@Delimiter{}}%
}
\newcommand\My@Paired@Delimiter[4]{%
  \newcommand#2{%
    \@ifstar{\start@PD{#1}{\delimitershortfall=-1sp}{#3}{#4}}
            {\start@PD{#1}{}{#3}{#4}}%
  }%
}
\newcommand\start@PD[5]{%
  #1\mathopen{\mathpalette\put@delim@helper{\put@delim{#2}{#3}{.}{#5}}}%
  #5%
  \mathclose{\mathpalette\put@delim@helper{\put@delim{#2}{.}{#4}{#5}}}%
}
\newcommand\put@delim@helper[2]{%
  \hbox{$\m@th\nulldelimiterspace=0pt #2#1$}%
}
\newcommand\put@delim[5]{%
  \setbox\z@\hbox{$\m@th#5{#4}$}%
  \setbox\tw@\null
  \ht\tw@\ht\z@ \dp\tw@\dp\z@
  #1#5%
  \left#2\box\tw@\right#3%
}

\makeatother
\MyPairedDelimiter*{\abs}{\lvert}{\rvert}
\MyPairedDelimiter*{\norm}{\lVert}{\rVert}
\MyPairedDelimiter{\set}{\{}{\}}

\providecommand{\scpr}[2]{\left( #1 \,\middle|\, #2 \right)}
\providecommand{\dupa}[2]{\left\langle #1 \mkern1.5mu,\mkern1.5mu #2 \right\rangle}
\renewcommand{\sp}{\scpr}

\newcommand\llim{
\mathchoice{\vcenter{\hbox{${\scriptstyle{-}}$}}}
{\vcenter{\hbox{$\scriptstyle{-}$}}}
{\vcenter{\hbox{$\scriptscriptstyle{-}$}}}
{\vcenter{\hbox{$\scriptscriptstyle{-}$}}}}
\newcommand\rlim{
\mathchoice{\vcenter{\hbox{${\scriptstyle{+}}$}}}
{\vcenter{\hbox{$\scriptstyle{+}$}}}
{\vcenter{\hbox{$\scriptscriptstyle{+}$}}}
{\vcenter{\hbox{$\scriptscriptstyle{+}$}}}}

\theoremstyle{plain} 
\newtheorem{theorem}{Theorem}[section]
\newtheorem{corollary}[theorem]{Corollary}
\newtheorem{lemma}[theorem]{Lemma}
\newtheorem{proposition}[theorem]{Proposition}
\theoremstyle{definition}

\newtheorem{definition}[theorem]{Definition}

\usepackage{enumitem}

\usepackage{hyperref}

\title{On solitary waves for the Korteweg--de Vries equation on metric star graphs}

\author{Delio Mugnolo \orcidlink{0000-0001-9405-0874}}
\address[D.M.]{
FernUniversit\"at in Hagen, Lehrgebiet Analysis, 58084 Hagen, Germany} 
\email{delio.mugnolo@fernuni-hagen.de}

\author{Diego Noja \orcidlink{0000-0003-1949-9369}}
\address[D.N.]{
Universit\`{a} degli Studi di Milano-Bicocca, Dipartimento di Matematica e Applicazioni, Via R.Cozzi 55, 20125, Milano, Italy} 
\email{diego.noja@unimib.it}

\author{Christian Seifert \orcidlink{0000-0001-9182-8687}}
\address[C.S.]{
Technische Universit\"at Hamburg, Institut f\"ur Mathematik, Am Schwarzenberg-Campus 3, 21073 Hamburg, Germany} \email{christian.seifert@tuhh.de}

\date{\today}

\begin{document}

\begin{abstract}
    We study the Korteweg--de Vries equation on a metric star graph and investigate existence of solitary waves on the metric graph in terms of the coefficients of the equation on each edge, the coupling condition at the central vertex of the star and the speeds of the travelling wave.
    We show that, with a continuity condition at the vertex, solitary waves can occur exactly when the parameters are chosen in a fairly special manner. We also consider coupling conditions beyond continuity.
\end{abstract}

\thanks{Declarations of interest: none.
Data Availability Statement: Data sharing not applicable to this article as no datasets were generated or analysed during the current study.\\
\medskip
This article is based upon work
from COST Action 18232 MAT-DYN-NET, supported by COST (European Cooperation in Science and
Technology), \url{www.cost.eu}
}

\maketitle

\section{Introduction}
\label{sec:intro}

The Korteweg--de Vries equation is a spatially one-dimensional model to describe shallow water waves in a narrow channel. Its investigation started in 1834 with experiments by Russell and was formally derived by Korteweg and de Vries in 1895 \cite{KortewegDeVries1895}; however, Boussinesq already introduced the equation in the 1870's.

Let $u$ be the function describing the height displacement of the water surface from equilibrium height, i.e.\ $u(t,x)$ is the deviation from equilibrium at time $t$ and position $x$. Then the Korteweg--de Vries (KdV) equation in dimensionless variables is given by
\[ \partial_t u = -\partial_x^3 - 6u\partial_x u.\]
However, note that there are various versions of this equation with different coefficients as well as first order terms $\partial_xu$ on the right-hand side, depending on rescaling and translation in the derivation process (the rescaling and translation can be well performed if the spatial domain is the whole real line).

In 1965, Zabusky and Kruskal \cite{ZabuskyKruskal1965} numerically discovered that solutions of the KdV equation admit solitary waves.

The linear part of the KdV equation, which also appears as the small amplitude long-wave approximation $u\partial_xu \approx 0$ gives rise to the so-called Airy equation (again, possibly with a first order term $\partial_x u$ on the right-hand side). This was studied by Airy as well as by Stokes \cite{Stokes1847} in order to understand the solitary waves observed by Russell.

The spatial domain for the KdV equation is given by an interval (or a halfline or the real line). When modelling branching channels, one thus needs to work with a bunch of intervals modelling the single channels as well as suitable coupling conditions at the junctions. Such geometric objects are called metric graphs and differential equations on metric graphs have been studied extensively during the last decades, see e.g.\ \cite{BerkolaikoKuchment2013, Mugnolo2014, Kurasov2024} and references therein for an overview on the topic. The linear part of KdV, i.e.\ the Airy equation, has been studied on metric graphs recently \cite{SobirovAkhmedovKarpovaJabbarova2015,SobirovUeckerAkhmedov2015b, SobirovUeckerAkhmedov2015, DeconinckSheilsSmith2016,MugnoloNojaSeifert2018, Seifert2018}; in fact, making use of the theory developed in \cite{SchubertSeifertVoigtWaurick2015}, in \cite{MugnoloNojaSeifert2018, Seifert2018} a characterisation of all coupling conditions giving rise to unitary or contractive dynamics for the Airy equation on metric (star) graphs were given.

As for the KdV equation on metric graphs, there are some well-posedness results in \cite{Cavalcante2018, AnguloPavaCavalcante2024} as well as stability results \cite{CavalcanteMunoz2019,AnguloPavaCavalcante2021, CavalcanteMunoz2023}.

In this short note we consider the KdV equation on a metric star graph with semi-infinite edges, i.e.\ on a bunch of halflines coupled at a single vertex, which we call $0$. For simplicity, we work with finitely many edges. We will consider general coefficients on the edges and include first-order terms, i.e.\ we consider
\[\partial_t u_e = -\alpha_e \partial_x^3 u + \beta_e \partial_x u + \gamma_e u_e\partial_x u_e\]
for each edge $e$ (where each single edge corresponds to a halfline).
We investigate existence of solitary waves for this equation on the metric star graph. In order to fix the model, we have to specify the coupling condition at the junction $0$. Here, we will work with linear coupling conditions coming from the linear part of the KdV equation, i.e.\ the Airy equation. One particular class of couplings stems from continuity of $u$ along the junction, but we will also consider couplings giving rise to jumps of $u$ at the junction (those are inspired by branching channels where there may be a large gradient at the junction, which can be ideally modelled by a jump). As it turns out, existence of solitary waves on the metric star graph is a rather rare situation; the coefficients $\alpha_e, \beta_e, \gamma_e$ as well as the wave speeds $c_e$ have to satisfy strong compatibility conditions in order to allow for solitary waves.

Let us outline the remaining part of the paper. In Section \ref{sec:KdV_R} we review the KdV equation (with general coefficients) on halflines and the real line, also explaining the solitary wave solution for the equation. Then in Section \ref{sec:Airy} we move on the the linear part, the Airy equation, on metric star graphs and collect the necessary results for the coupling conditions at the junction. We consider the KdV equation on metric star graphs in Section \ref{sec:KdV_graphs} and characterize existence of solitary waves by means of compatibility conditions for the coefficients.

\section{Solitary waves for Korteweg--de Vries equation on halflines and the real line}
\label{sec:KdV_R}

The results stated in this section are essentially contained in \cite{KortewegDeVries1895}. However, for self-consistency, for the reader's convenience and to fix notation we include full statements and proofs.

Let $\Omega\in \set{(-\infty,0), (0,\infty), \R}$, $\alpha, \beta,\gamma\in\R$ with $\alpha>0$ and $\gamma\neq 0$. Then the \emph{Korteweg--de Vries (KdV) equation} (on $\Omega$) is given by
\begin{equation}
\label{eq:KdV}
  \partial_t u = -\alpha \partial_x^3 u + \beta \partial_x u + \gamma u \partial_x u \quad\text{on} \quad \R\times \Omega,
\end{equation}
where the standard choice of coefficients is $\alpha=1$, $\beta=0$, and $\gamma=-6$.

\begin{definition}
  We say that $u\from \R\times \Omega\to\R$ is a \emph{travelling wave} for the KdV equation \eqref{eq:KdV}
  provided there exist $\varphi\in C^3(\R)$ and $c\in \R\setminus\{0\}$ such that 
  $u\from \R\times \Omega\to\R$ defined by $u(t,x):= \varphi(x-ct)$ for all $t\in\R$ and $x\in \Omega$ solves \eqref{eq:KdV}.
  Moreover, if $\varphi$ admits at most one local extremum and the limits $\lim_{y\to\pm\infty} \varphi(y)$ exist then $u$ determined by $\varphi$ is called \emph{solitary wave}.
\end{definition}

We want to find $\varphi\from\R\to\R$ such that $u\from \R\times\Omega \to\R$ defined by 
\[u(t,x):=\varphi(x-ct) \quad(x\in \Omega, t\in \R),\]
where $c\in\R$, is a solitary wave.

Plugging the description of $u$ into the equation we obtain
\[-\alpha \varphi''' + (\beta+c)\varphi' + \gamma \varphi \varphi' = 0.\]
Integrating this equation, we get
\[-\alpha \varphi'' + (\beta+c)\varphi + \frac{\gamma}{2} \varphi^2 = A\]
for some $A\in\R$. The corresponding first order system reads, with $\psi:=\varphi'$,
\begin{align}
\label{eq:first_order_system}
\begin{split}
  \varphi' & = \psi, \\
  \psi' & = \frac{-A+(\beta+c)\varphi + \frac{\gamma}{2}\varphi^2}{\alpha}.
\end{split}
\end{align}

Define
\[H(\varphi,\psi) := \frac{\psi^2}{2} - \frac{1}{\alpha} \bigl(-A\varphi + \frac{\beta+c}{2} \varphi^2 + \frac{\gamma}{6}\varphi^3\bigr).\]
Then $H$ is constant along solutions of the first order system.

Let us compute the stationary points of the first order system, i.e.\
\begin{align*}
  \varphi' & = 0, \\
  \psi' & = 0.
\end{align*}
Thus, $\psi = 0$ and $-A+(\beta+c)\varphi + \frac{\gamma}{2} \varphi^2 = 0$. Since $\gamma\neq 0$, we obtain
\begin{equation}
\label{eq:quadratic_phi}
\varphi^2 + \frac{2(\beta+c)}{\gamma} \varphi - \frac{2A}{\gamma} = 0. 
\end{equation}

\begin{lemma}
  If $(\beta+c)^2+2A\gamma< 0$ then there is no stationary point and all trajectories of solutions of \eqref{eq:first_order_system} are unbounded.
\end{lemma}

\begin{proof}
    Without loss of generality, let $\gamma>0$. Then the minimal value of $\varphi \mapsto \frac{-A+(\beta+c)\varphi + \frac{\gamma}{2}\varphi^2}{\alpha}$ is given for $\varphi = -\frac{\beta+c}{\gamma}$ with value
    \[\frac{-A-\frac{1}{2\gamma}(\beta+c)^2}{\alpha} = \frac{1}{2\alpha\gamma}\bigl(-2A\gamma-(\beta+c)^2\bigr).\]
    In view of \eqref{eq:first_order_system}, $\psi'$ is bounded from below by a positive constant, so $\psi$ will be unbounded and therefore also $\varphi$ will be unbounded.
\end{proof}

\begin{lemma}
  If $(\beta+c)^2+2A\gamma = 0$ then there is exactly one stationary point
  \[p = \Bigl(-\frac{\beta+c}{\gamma},0\Bigr).\]
\end{lemma}

\begin{proof}
    Under the stated condition there is exactly one real solution of the quadratic equation \eqref{eq:quadratic_phi}, namely $\varphi = -\frac{\beta+c}{\gamma}$, which corresponds to the stationary point $p = \Bigl(-\frac{\beta+c}{\gamma},0\Bigr)$. 
\end{proof}

\begin{lemma}
\label{lem:two_stationary_points}
  If $(\beta+c)^2+2A\gamma> 0$ then there are two stationary points
  \begin{align*}
    p_- & = \Bigl(\frac{ -(\beta+c) - \sqrt{(\beta+c)^2 + 2A\gamma}}{\gamma},0\Bigr),\\
    p_+ & = \Bigl(\frac{ -(\beta+c) + \sqrt{(\beta+c)^2 + 2A\gamma}}{\gamma},0\Bigr).
  \end{align*}
  Moreover, $p_-$ is a center an $p_+$ is a saddle point.
\end{lemma}

\begin{proof}
  Under the stated condition the quadratic equation \eqref{eq:quadratic_phi} for $\varphi$ has two real solutions, namely
  \[\varphi_{1,2} = -\frac{\beta+c}{\gamma} \pm \sqrt{\bigl(\frac{\beta+c}{\gamma}\bigr)^2 + \frac{2A}{\gamma}} = \frac{ -(\beta+c) \pm \sqrt{(\beta+c)^2 + 2A\gamma}}{\gamma}.\]
  Since $\psi=0$, we obtain the two stationary points.
  
  Note that the right-hand side of the first order system \eqref{eq:first_order_system} is given by $f(\varphi,\psi) := \begin{pmatrix} \psi \\ \frac{-A+(\beta+c)\varphi + \frac{\gamma}{2}\varphi^2}{\alpha}\end{pmatrix}$.
  The Jacobian of $f$ at $p_-$ has the eigenvalues $\lambda$ determined by
  \[\lambda^2 = -\frac{1}{\alpha} \sqrt{(\beta+c)^2+2A\gamma} <0.\]
  Thus, there are two conjugate complex eigenvalues on the imaginary axis, and so $p_-$ is a center.  
  The Jacobian of $f$ at $p_+$ has the eigenvalues $\lambda$ determined by
  \[\lambda^2 = +\frac{1}{\alpha} \sqrt{(\beta+c)^2+2A\gamma}>0.\]
  Thus, there is one positive and one negative eigenvalue and $p_+$ is a saddle point.
\end{proof}

\begin{lemma}
\label{lem:homoclinic_orbit}
  Let $(\beta+c)^2+2A\gamma> 0$. Then there exists a unique solution $(\varphi,\psi)$ of the first order system \eqref{eq:first_order_system} such that
  \[\lim_{y\to\pm\infty} \varphi(y) = \frac{ -(\beta+c) + \sqrt{(\beta+c)^2 + 2A\gamma}}{\gamma}, \qquad \lim_{y\to\pm\infty} \psi(y) = 0.\]
\end{lemma}

\begin{proof}
  By Lemma \ref{lem:two_stationary_points}, there exists a unique homoclinic orbit (for $p_+$) which corresponds to the unique solution $y\mapsto (\varphi(y),\psi(y))$ approximating $p_+$ at $\pm\infty$, see also Figure \ref{fig:phase}.
\end{proof}

\begin{figure}[htb]
    \centering
    \begin{tikzpicture}
        \begin{axis}[
        xmin = -4, xmax = 4,
        ymin = -4, ymax = 4,
        zmin = 0, zmax = 1,
        axis equal image,
        view = {0}{90},
        samples=31,
        xlabel={$\varphi$},
        ylabel={$\psi$},
        xticklabels=\empty,
        yticklabels=\empty,
        ]
        \addplot3[
            quiver = {
                u = {y/sqrt(y^2+ (x+0.5*x^2)^2)},
                v = {(x+0.5*x^2)/sqrt(y^2+ (x+0.5*x^2)^2)},
                scale arrows = 0.25,
            },
            -stealth,
            domain = -4:4,
            domain y = -4:4,
        ] {0};
        \addplot coordinates {
        (-2,0)
        };
        \addplot coordinates {
        (0,0)
        };
        \addplot[domain=-3:0, blue, thick]{
            sqrt(2*(0.5*x^2+1/6*x^3))
        };
        \addplot[domain=-3:0, blue, thick]{
            -sqrt(2*(0.5*x^2+1/6*x^3))
        };
        \addplot[domain=0:3, blue, thick]{
            sqrt(2*(0.5*x^2+1/6*x^3))
        };
        \addplot[domain=0:3, blue, thick]{
            -sqrt(2*(0.5*x^2+1/6*x^3))
        };
\end{axis}

    \end{tikzpicture}
\caption{Phase portrait for \eqref{eq:first_order_system}.}
\label{fig:phase}
\end{figure}
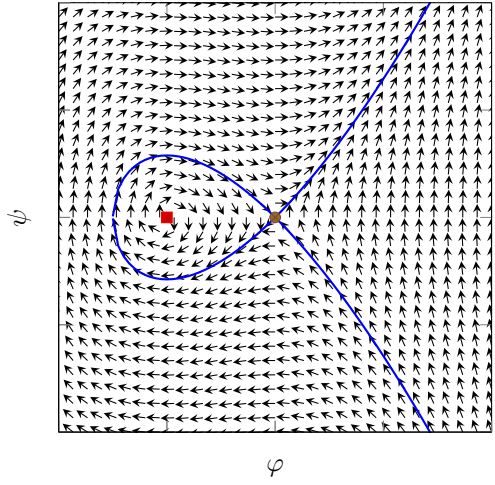

In view of Lemma \ref{lem:homoclinic_orbit}, in order to obtain a solitary wave we have to assume $A=0$ and $\beta+c>0$; we therefore specialize ot this case.

\begin{lemma}
\label{lem:sign}
  Let $A=0$ and $\beta+c>0$. Let $(\varphi,\psi)$ be the unique solution of the first order system according to Lemma \ref{lem:homoclinic_orbit}. Then:
  \begin{enumerate}
    \item
      $(\varphi,\psi)$ is given by 
      \[\psi^2 = \frac{\beta+c}{\alpha} \varphi^2 + \frac{\gamma}{3\alpha}\varphi^3\]
      and
      \[\lim_{y\to\pm\infty} \varphi(y) = \lim_{y\to\pm\infty} \psi(y) = 0.\]
    \item
      There exists a unique $y_0\in\R$ such that $\varphi(y_0) = -\frac{3}{\gamma}(\beta+c)$ and $\psi(y_0) = 0$. 
    \item 
      Let $\gamma>0$. Then $\varphi(y)<0$ for all $y\in\R$, $\varphi$ is decreasing on $(-\infty,y_0)$ and $\varphi$ is increasing on $(y_0,\infty)$.
    \item
      Let $\gamma<0$. Then $\varphi(y)>0$ for all $y\in\R$, $\varphi$ is increasing on $(-\infty,y_0)$ and $\varphi$ is decreasing on $(y_0,\infty)$.
  \end{enumerate}
\end{lemma}

\begin{proof}
  (a)
  By Lemma \ref{lem:homoclinic_orbit}, there exists a unique solution $(\varphi,\psi)$, and $\varphi$ and $\psi$ have the stated limits at $\pm\infty$.
  Since $A=0$ we have $p_+ = \bigl(0,0\bigr)$. Moreover, as $H$ is constant along trajectories, for the unique solution $(\varphi,\psi)$ describing the homoclinic orbit we obtain $H(\varphi,\psi) = 0$ and therefore
  \[\psi^2 =  \frac{\beta+c}{\alpha} \varphi^2 + \frac{\gamma}{3\alpha}\varphi^3.\]
  
  (b)
  Let $y\in\R$ such that $\psi(y) = 0$. Then 
  \[0 = \frac{\beta+c}{\alpha} \varphi(y)^2 + \frac{\gamma}{3\alpha}\varphi(y)^3,\]
  and therefore $\varphi(y) = 0$ or $\varphi(y) = -\frac{3}{\gamma}(\beta+c)$.
  Assume there are $y_1,y_2\in\R$, $y_1< y_2$ such that $\varphi(y_1)=\varphi(y_2) = -\frac{3}{\gamma}(\beta+c)$ and $\psi(y_1)=\psi(y_2) = 0$. 
  Since $\lim_{y\to\pm\infty} \varphi(y) = 0$ and the first order system \eqref{eq:first_order_system} is uniquely solvable given the values $\varphi(y)$ and $\psi(y)$ of $\varphi$ and $\psi$ at $y\in \R$, then necessarily
  $\varphi(y) = -\frac{3}{\gamma}(\beta+c)$ and $\psi(y) = 0$ for all $y\in (y_1,y_2)$ (otherwise we would have a periodic orbit, contradicting the limit as $y\to \pm\infty$).
  Exploiting the first order system \eqref{eq:first_order_system} then yields
  \[\varphi'(y) = 0,\qquad \psi'(y) = \frac{3(\beta+c)^2}{2\alpha\gamma} \neq 0 \quad(y\in (y_1,y_2)),\]
  a contradiction.  
  
  (c)
  Let $\gamma>0$. Then the first component of $p_-$ is negative. 
  The unstable direction for the Jacobian of $f$ at $p_+$ is given by $\lin \begin{pmatrix} 1\\\sqrt{\frac{1}{\alpha}(\beta+c)}\end{pmatrix}$.
  Therefore, the homoclinic orbit describing the unique solution leaves $p_+$ in direction $\begin{pmatrix} -1\\ -\sqrt{\frac{1}{\alpha}(\beta+c)}\end{pmatrix}$.
  Hence, $\varphi(y)<0$ for all $y\in (-\infty,y_0)$ and since $\psi = \varphi'$ is negative on $(-\infty,y_0)$, $\varphi$ is decreasing there. 
  Analogously, the stable direction for the Jacobian of $f$ at $p_+$ is given by $\lin \begin{pmatrix} 1\\ -\sqrt{\frac{1}{\alpha}(\beta+c)}\end{pmatrix}$. 
  Therefore, the homoclinic orbit describing the unique solution enters $p_+$ in direction $\begin{pmatrix} -1\\ \sqrt{\frac{1}{\alpha}(\beta+c)}\end{pmatrix}$.
  Thus, $\varphi(y)>0$ for all $y\in (y_0,\infty)$ and since $\psi = \varphi'$ is positive on $(y_0,\infty)$, $\varphi$ is increasing there.  
  
  (d)
  The case $\gamma<0$ is proved analogously to the case $\gamma>0$.
\end{proof}


For the next theorem note that $\sech(\cdot) = \frac{1}{\cosh(\cdot)}$ by definition.

\begin{theorem}
\label{thm:KdV_R_solitary_wave}
  Let $A=0$ and $\beta+c>0$.
  Then \eqref{eq:KdV} admits a solitary wave solution determined by
  \[\varphi(y) = -\frac{3(\beta+c)}{\gamma}\sech^2\Bigl(\frac{\sqrt{\frac{\beta+c}{\alpha}}}{2}(y-y_0)\Bigr)\qquad(y\in\R).\]
\end{theorem}

\begin{proof}
  Let $\gamma>0$. 
  By Lemma \ref{lem:sign}, for $y<y_0$ we have $\varphi(y)<0$ and $\psi(y)<0$, and moreover $\psi(y) = -\sqrt{\frac{\beta+c}{\alpha} \varphi(y)^2 + \frac{\gamma}{3\alpha}\varphi(y)^3}$.
  Taking into account the facts that $\psi = \varphi'$ and $\varphi(y_0) = -\frac{3}{\gamma}(\beta+c)$, applying separation of variables and integrating, we observe
  \[\int_{\varphi(y)}^{-\frac{3}{\gamma}(\beta+c)} - \sqrt{\frac{1}{\frac{\beta+c}{\alpha} \varphi^2 + \frac{\gamma}{3\alpha}\varphi^3}}\,d\varphi = \int_y^{y_0} \,dy = y_0-y.\]
  Integrating and solving for $\varphi(y)$ yields
  \[\varphi(y) = -\frac{3(\beta+c)}{\gamma}\sech^2\Bigl(\frac{\sqrt{\frac{\beta+c}{\alpha}}}{2}(y_0-y)\Bigr) \qquad(y<y_0).\]
  Analogously as above, Lemma \ref{lem:sign} yields for $y>y_0$ that $\varphi(y)<0$ and $\psi(y)>0$, and moreover $\psi(y) = \sqrt{\frac{\beta+c}{\alpha} \varphi(y)^2 + \frac{\gamma}{3\alpha}\varphi(y)^3}$.
  Since $\psi=\varphi'$ and $\varphi(y_0) = -\frac{3}{\gamma}(\beta+c)$, separation of variables yields
  \[\int_{-\frac{3}{\gamma}(\beta+c)}^{\varphi(y)} \sqrt{\frac{1}{\frac{\beta+c}{\alpha} \varphi^2 + \frac{\gamma}{3\alpha}\varphi^3}}\,d\varphi = \int_{y_0}^y \,dy = y-y_0.\]
  Integrating and solving for $\varphi(y)$ we obtain
  \[\varphi(y) = -\frac{3(\beta+c)}{\gamma}\sech^2\Bigl(\frac{\sqrt{\frac{\beta+c}{\alpha}}}{2}(y-y_0)\Bigr) \qquad(y>y_0).\]
  Since $\sech(x) = \sech(-x)$ for all $x\in\R$, $\sech(0) = 1$ and $\varphi(y_0) = -\frac{3(\beta+c)}{\gamma}$ we thus obtain
  \[\varphi(y) = -\frac{3(\beta+c)}{\gamma}\sech^2\Bigl(\frac{\sqrt{\frac{\beta+c}{\alpha}}}{2}(y-y_0)\Bigr) \qquad(y\in\R).\]
  
  Now, let $\gamma<0$. Then analogously we obtain
  \[\varphi(y) = -\frac{3(\beta+c)}{\gamma}\sech^2\Bigl(\frac{\sqrt{\frac{\beta+c}{\alpha}}}{2}(y-y_0)\Bigr)\qquad(y\in\R). \qedhere\]
\end{proof}

\section{The Airy equation on metric star graphs}
\label{sec:Airy}

In this section we consider the linear part of the Korteweg--de Vries equation on metric star graphs with semi-infinite edges.

Consider star graph with halfline edges attached at the vertex $0$. Let $E=E_-\cup E_+$ be finite, where $e\in E_-$ is an edge parametrized by the interval $(-\infty,0)$, and $e\in E_+$ is an edge parametrized $(0,\infty)$.
For $e\in E$ let $\alpha_e>0$ and $\beta_e\in \R$. 

\begin{figure}[htb]
    \centering
    \begin{tikzpicture}[scale=0.7,>=stealth]
        \draw[fill] (0,0) circle(0.1) node[below]{$0$};
        \draw (0,0)--(-3,-2);
        \draw[->] (-3,-2)--(-1.5,-1);
        \draw[dashed] (-3,-2)--(-6,-4);
        \draw (0,0)--(-3,-1);
        \draw[->] (-3,-1)--(-1.5,-0.5);
        \draw[dashed] (-3,-1)--(-6,-2);
        \draw (0,0)--(-3,-0);
        \draw[->] (-3,-0)--(-1.5,-0);
        \draw[dashed] (-3,-0)--(-6,-0);
        \draw (0,0)--(-3,1);
        \draw[->] (-3,1)--(-1.5,0.5);
        \draw[dashed] (-3,1)--(-6,2);
        \draw (0,0)--(-3,2);
        \draw[->] (-3,2)--(-1.5,1);
        \draw[dashed] (-3,2)--(-6,4);
        \draw (0,0)--(3,-2);
        \draw[->] (0,0)--(1.5,-1);
        \draw[dashed] (3,-2)--(6,-4);
        \draw (0,0)--(3,-1);
        \draw[->] (0,0)--(1.5,-0.5);
        \draw[dashed] (3,-1)--(6,-2);
        \draw (0,0)--(3,1);
        \draw[->] (0,0)--(1.5,0.5);
        \draw[dashed] (3,1)--(6,2);
        \draw (0,0)--(3,2);
        \draw[->] (0,0)--(1.5,1);
        \draw[dashed] (3,2)--(6,4);
    \end{tikzpicture}
    \caption{A metric star graph with $|E_-|=5$ and $|E_+| = 4$.}
\label{fig:stargraph}
\end{figure}
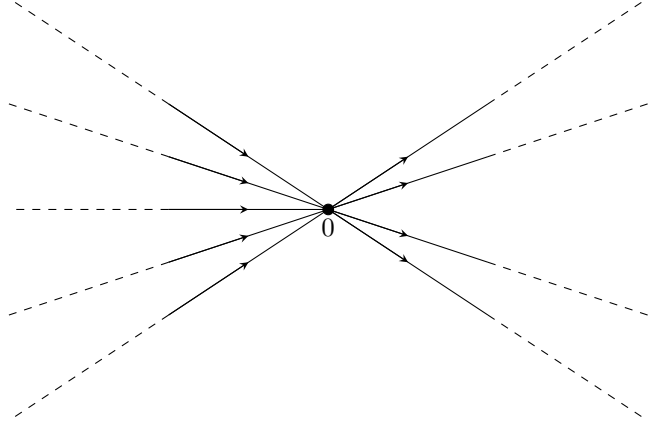

We consider the \emph{Airy equation}
\[\partial_t u_e = -\alpha_e \partial_x^3 u_e + \beta_e \partial_x u_e \quad\text{on $\R\times \Omega_e$},\]
where $\Omega\in\set{(-\infty,0),(0,\infty)}$ corresponds to the parametrization of $e$ and $e\in E$.
Note that the Airy equation can (also) be deduced from the KdV equation by setting $\gamma=0$ (which was formally not allowed above) or linearizing the equation around the solution $u=0$.
In any case, the Airy equation yields the linear part of the KdV equation.

We follow the arguments in \cite{MugnoloNojaSeifert2018,Seifert2018} (note that we changed the sign in front of the $\alpha_e$'s compared to the references). All vector spaces appearing are over the field $\K\in\{\R,\C\}$.

Define the minimal operator $A_0$ in $\calH:=\bigoplus_{e\in E_-} L_2(-\infty,0) \oplus \bigoplus_{e\in E_+} L_2(0,\infty)$ by
\begin{align*}
  D(A_0) & := \bigoplus_{e\in E_-} C_c^\infty(-\infty,0) \oplus \bigoplus_{e\in E_+} C_c^\infty(0,\infty),\\
  A_0 u & := (-\alpha_e u_e''' + \beta_e u_e')_{e\in E} \quad(u\in D(A_0)).
\end{align*}
Then an easy computation yields
\begin{align*}
  D(A_0^*) & = \bigoplus_{e\in E_-} H^3(-\infty,0) \oplus \bigoplus_{e\in E_+} H^3(0,\infty),\\
  A_0^* u & := (\alpha_e u_e''' - \beta_e u_e')_{e\in E} \quad(u\in D(A_0)),
\end{align*}
and hence $A_0$ is skew-symmetric, i.e.\ densely defined and $A_0\subseteq -A_0^*$.

We search for operators $A_0\subseteq A\subseteq -A_0^*$ such that the abstract Cauchy problem
\begin{align*}
  u'(t) & = Au(t) \quad(t>0),\\
  u(0) & = u_0\in \calH
\end{align*}
is well-posed, i.e.\ $A$ is the generator of a $C_0$-semigroup. For this we will have to impose coupling conditions at the vertex $0$.
Note that this abstract Cauchy problem is a functional analytic description of the Airy equation on each edge $e\in E$ including the coupling conditions at $0$.

For $u\in D(A_0^*)$ let us abbreviate
\[u^{(k)}(0\llim):=\bigl(u_e^{(k)}(0\llim)\bigr)_{e\in E_-}, \quad u^{(k)}(0\rlim):=\bigl(u_e^{(k)}(0\rlim)\bigr)_{e\in E_+} \quad(k\in\set{0,1,2}).\]
Moreover, let us denote by $\alpha_\pm$ and $\beta_\pm$ the multiplication operators on $\K^{E_\pm}$ defined by
\[\alpha_\pm x := (\alpha_e x_e)_{e\in E_\pm},\quad \beta_\pm x:= (\beta_e x_e)_{e\in E_\pm} \quad(x\in \K^{E_\pm}).\]
Define
\[\calG_\pm:=\K^{E_\pm}\oplus\K^{E_\pm}\oplus\K^{E_\pm},\]
and the block operator matrices $B_\pm$ on $\calG_\pm$ by
\[B_\pm  = \begin{pmatrix}
	    -\beta_\pm & 0 & \alpha_\pm \\
	    0 & -\alpha_\pm & 0 \\
	    \alpha_\pm & 0 & 0
           \end{pmatrix}.\]
Since $(\alpha_e)_{e\in E}$, $(1/\alpha_e)_{e\in E}$ and $(\beta_e)_{e\in E}$ are bounded we have $B_\pm\in \calL(\calG_\pm)$, $B_\pm$ is injective and $B_\pm^{-1}\in\calL(\calG_\pm)$.
Define an indefinite inner product $\dupa{\cdot}{\cdot}_\pm\from \calG_\pm\times\calG_\pm\to \K$ by 
\[\dupa{x}{y}_\pm := \sp{B_\pm x}{y}_{\calG_\pm} \quad(x,y\in\calG_\pm).\]
Then $\calK_\pm:=\bigl(\calG_\pm, \dupa{\cdot}{\cdot}_\pm\bigr)$ are Krein spaces.

For a linear operator $L\from \calK_- \to\calK_+$ we define
\begin{align*}
  D(A_L) & := \set{u\in D(A_0^*);\; L\bigl(u(0\llim),u'(0\llim),u''(0\llim)\bigr) = \bigl(u(0\rlim),u'(0\rlim),u''(0\rlim)\bigr)},\\
  A_L u & := -A_0^* u \quad(u\in D(A_L)).
\end{align*}

\begin{proposition}[{see \cite[Theorem 3.9 and Theorem 3.16]{MugnoloNojaSeifert2018}}]
\label{prop:A_L}
  Let $L\from \calK_- \to\calK_+$ be a linear operator. Then:
  \begin{enumerate}
    \item
      $A_L$ generates a unitary $C_0$-group if and only if $L$ is $(\calK_-,\calK_+)$-unitary, i.e.
      \[\dupa{Lx}{Ly}_+  = \dupa{x}{y}_- \quad(x,y\in\calK_-).\]
    \item
      $A_L$ generates a contractive $C_0$-semigroup if and only if $L$ is $(\calK_-,\calK_+)$-contractive, i.e.\
      \[\dupa{Lx}{Lx}_+  \leq \dupa{x}{x}_- \quad(x\in\calK_-),\]
      and $L^\sharp$ is $(\calK_+,\calK_-)$-contractive, i.e.\
      \[\dupa{L^\sharp x}{L^\sharp x}_- \leq \dupa{x}{x}_+ \quad(x\in\calK_+),\]
      where $L^\sharp\from \calK_+\to \calK_-$ is the $(\calK_-,\calK_+)$-adjoint of $L$ given by
      \[\dupa{Lx}{y}_{+} = \dupa{x}{L^\sharp y}_- \quad(x\in \calK_-,y\in\calK_+).\]
  \end{enumerate} 
\end{proposition}

A special case of coupling conditions at $0$ stems from separating the boundary values of the first derivatives. In order to formulate this, let $Y\subseteq \K^{E_-}\oplus \K^{E_+}$ be a subspace and $U\from (\K^{E_+},(\alpha_e)_{e\in E_+}) \to (\K^{E_-},(\alpha_e)_{e\in E_-})$ be linear. We define
\begin{align*}
  D(A_{Y,U}) & := \Bigl\{u\in D(A_0^*);\; \bigl(u(0\llim),u(0\rlim)\bigr)\in Y,\, \\
    & \qquad\qquad\qquad\qquad \begin{pmatrix} -\alpha_- & 0 \\ 0 & \alpha_+ \end{pmatrix}\begin{pmatrix} u''(0\llim) \\ u''(0\rlim)\end{pmatrix} + \begin{pmatrix} \frac{\beta_-}{2} & 0 \\ 0 & -\frac{\beta_+}{2} \end{pmatrix}\begin{pmatrix} u(0\llim) \\ u(0\rlim)\end{pmatrix} \in Y^\bot,\\
  & \qquad\qquad\qquad\qquad u'(0\llim) = U u'(0\rlim)\Bigr\},\\
  A_{Y,U} u & := -A_0^* u \quad(u\in D(A_{Y,U})).
\end{align*}

\begin{proposition}[{compare \cite[Corollary 3.20 and Theorem 3.23]{MugnoloNojaSeifert2018}}]
\label{prop:A_Y,U}
  Let $Y\subseteq \K^{E_-}\oplus \K^{E_+}$ a subspace and $U\from (\K^{E_+},(\alpha_e)_{e\in E_+}) \to (\K^{E_-},(\alpha_e)_{e\in E_-})$ linear. Then:
  \begin{enumerate}
    \item
      $A_{Y,U}$ generates a unitary $C_0$-group if and only if $U$ is unitary.
    \item
      $A_{Y,U}$ generates a contractive $C_0$-semigroup if and only if $U$ is a contraction.
  \end{enumerate}
\end{proposition}

\section{The KdV equation on metric star graphs}
\label{sec:KdV_graphs}

As in Section \ref{sec:Airy}, we consider a metric star graph with halflines attached at the vertex $0$. 
Let $E=E_-\cup E_+$ be finite, where $e\in E_-$ is an edge parametrized by the interval $(-\infty,0)$, and $e\in E_+$ is an edge parametrized $(0,\infty)$.
For $e\in E$ let $\alpha_e,\beta_e,\gamma_e\in \R$ with $\alpha_e>0$ and $\gamma_e\neq 0$. 
 
In this section we consider the KdV equation on the metric star graph, i.e.
\begin{equation}
\label{eq:KdV_star_graph}
  \partial_t u_e = -\alpha_e \partial_x^3 u_e + \beta_e \partial_x u_e + \gamma_e u_e\partial_x u_e \quad\text{on $\R\times \Omega_e$} \quad(e\in E),
\end{equation}
where $\Omega_E :=(-\infty,0)$ for $e\in E_-$ and $\Omega_e:=(0,\infty)$ for $e\in E_+$.

\begin{definition}
  We say that $u = (u_e)_{e\in E}$ is a \emph{travelling wave} for the KdV equation \eqref{eq:KdV_star_graph} on the star graph provided for $e\in E$ there exist $\varphi_e\in C^3(\R)$ and $c_e\in\R\setminus\{0\}$ such that $u = (u_e)_{e\in E}$ with
  $u_e\from \R\times \Omega_e\to\R$ defined by $u_e(t,x) := \varphi_e(x-c_e t)$ ($t\in\R$, $x\in\Omega_e$) solves \eqref{eq:KdV_star_graph}.
  Moreover, if for all $e\in E$ we have that $\varphi_e$ admits at most one local extremum and the limits $\lim_{y\to \pm\infty} \varphi_e(y)$ exist then $u = (u_e)$ determined by $(\varphi_e)$ is called \emph{solitary wave}.
\end{definition}

We now combine the results reviewed in Sections \ref{sec:KdV_R} and \ref{sec:Airy}, i.e.\ we look for solitary waves for \eqref{eq:KdV_star_graph} satisfying suitable coupling conditions at the vertex $0$.
According to Theorem \ref{thm:KdV_R_solitary_wave}, for $e\in E$ let $\varphi_e\in C^3(\R)$ be defined by
\[\varphi_e(y):=-\frac{3(\beta_e+c_e)}{\gamma_e}\sech^2\Bigl(\frac{\sqrt{\frac{\beta_e+c_e}{\alpha_e}}}{2}(y-y_{0,e})\Bigr)\qquad(y\in\R),\]
where $c_e\in\R$ and $y_{0,e}\in\R$.

For $u\in \bigoplus_{e\in E_-} H^3(-\infty,0) \oplus \bigoplus_{e\in E_+} H^3(0,\infty)$ define
\[u^{(k)}(0\llim):=\bigl(u_e^{(k)}(0\llim)\bigr)_{e\in E_-}, \quad u^{(k)}(0\rlim):=\bigl(u_e^{(k)}(0\rlim)\bigr)_{e\in E_+} \quad(k\in\set{0,1,2}).\]
As above, let us denote by $\alpha_\pm$ and $\beta_\pm$ the multiplication operators on $\K^{E_\pm}$ defined by
\[\alpha_\pm x := (\alpha_e x_e)_{e\in E_\pm},\quad \beta_\pm x:= (\beta_e x_e)_{e\in E_\pm} \quad(x\in \K^{E_\pm}).\]
Note that if we have a travelling wave for \eqref{eq:KdV_star_graph} with $\varphi_e$ as above, then 
\[u(t,\cdot) = (u_e(t,\cdot))_{e\in E} \in \bigoplus_{e\in E_-} H^3(-\infty,0) \oplus \bigoplus_{e\in E_+} H^3(0,\infty) \quad(t\in\R).\] 

We impose coupling conditions in the spirit of Proposition \ref{prop:A_Y,U}. Specifically, we first consider continuity conditions at the vertex $0$, which can be written as $Y:=\lin(\1)$ being the one-dimensional vector space of vectors having equal entries.

\begin{theorem}
  Let $Y:=\lin(\1) \subseteq \K^{E_-}\oplus \K^{E_+}$, and $U\from(\K^{E_+},(\alpha_e)_{e\in E_+}) \to (\K^{E_-},(\alpha_e)_{e\in E_-})$ contractive.
  Consider \eqref{eq:KdV_star_graph} with coupling conditions
  \begin{align*}
    \bigl(u(t,0\llim),u(t,0\rlim)\bigr) & \in Y, \\
    \begin{pmatrix} -\alpha_- & 0 \\ 0 & \alpha_+ \end{pmatrix}\begin{pmatrix} \partial_x^2 u(t,0\llim) \\ \partial_x^2 u(t,0\rlim)\end{pmatrix} + \begin{pmatrix} \frac{\beta_-}{2} & 0 \\ 0 & -\frac{\beta_+}{2} \end{pmatrix}\begin{pmatrix} u(t,0\llim) \\ u(t,0\rlim)\end{pmatrix} & \in Y^\bot,\\
    \partial_x u(t,0\llim) & = U \partial_x u(t,0\rlim)
  \end{align*}
  for all $t\in\R$.
  Then \eqref{eq:KdV_star_graph} admits a solitary wave if and only if
  \[U\bigl(\frac{1}{c_e}\bigr)_{e\in E_+} = \bigl(\frac{1}{c_e}\bigr)_{e\in E_-}\]
  and for all $e,\tilde{e}\in E$ we have
  \begin{align*}
    \beta_e+c_e & > 0,\\
    \frac{y_{0,e}}{c_e} & = \frac{y_{0,\tilde{e}}}{c_{\tilde{e}}},\\
    \sqrt{\frac{\beta_e+c_e}{\alpha_e}} c_e & = \sqrt{\frac{\beta_{\tilde{e}}+c_{\tilde{e}}}{\alpha_{\tilde{e}}}} c_{\tilde{e}},\\
    \frac{(\beta_e+c_e)}{\gamma_e} & = \frac{(\beta_{\tilde{e}}+c_{\tilde{e}})}{\gamma_{\tilde{e}}},\\
    \sum_{{e}\in E_-}  \frac{\alpha_{{e}}}{c_{{e}}^2} - \sum_{{e}\in E_+}  \frac{\alpha_{{e}}}{c_{{e}}^2} & = 0 ,\\
    \sum_{e\in E_-} \beta_e - \sum_{e\in E_+} \beta_e & = 0.
  \end{align*}
  The solitary wave is then given by
  \[\varphi_e(y):=-\frac{3(\beta_e+c_e)}{\gamma_e}\sech^2\Bigl(\frac{\sqrt{\frac{\beta_e+c_e}{\alpha_e}}}{2}(y-y_{0,e})\Bigr)\qquad(y\in\R, e\in E).\]
\end{theorem}
Note that the condition
\[U\bigl(\frac{1}{c_e}\bigr)_{e\in E_+} = \bigl(\frac{1}{c_e}\bigr)_{e\in E_-}\]
fits well to the fifth assumption on the coefficients, since
\[\norm{U\bigl(\frac{1}{c_e}\bigr)_{e\in E_+}}_{\alpha_-}^2 = \norm{\bigl(\frac{1}{c_e}\bigr)_{e\in E_-}}_{\alpha_-}^2 = \sum_{e\in E_-} \abs{\frac{1}{c_e}}^2 \alpha_e = \sum_{e\in E_+} \abs{\frac{1}{c_{e}}}^2 \alpha_{e} = \norm{\bigl(\frac{1}{c_e}\bigr)_{e\in E_+}}_{\alpha_+}^2.\]
So $U$ can indeed be unitary or contractive, and preserves the weighted norm for inverse wave speeds.

\begin{proof} 
  $\Longrightarrow$:
  A solitary wave on the metric star graph is also a solitary wave for each edge $e\in E$ and is thus given by
  \[\varphi_e(y):=-\frac{3(\beta_e+c_e)}{\gamma_e}\sech^2\Bigl(\frac{\sqrt{\frac{\beta_e+c_e}{\alpha_e}}}{2}(y-y_{0,e})\Bigr)\qquad(y\in\R, e\in E).\]
  Using the continuity condition, we obtain that there exist $\varphi\in C^3(\R)$ and $k_e\in\R$ for all $e\in E$ such that
  \[\varphi_e(t) = \varphi(k_e t) \quad(t\in\R).\] 
  We easily see from the continuity condition that $k_e c_e = k_{\tilde{e}} c_{\tilde{e}}$ for all $e,\tilde{e}\in E$. 
  Inspecting the coupling conditions and comparing the coefficients yield the conditions.
  
  $\Longleftarrow$:
  By the first assumption on the coefficients and Theorem \ref{thm:KdV_R_solitary_wave} we obtain solitary waves described by $\varphi_e$ on each edge $e\in E$.
  It remains to show that these $\varphi_e$ can be coupled at the vertex is a way that fits to the coupling conditions.
  
  Exploiting the continuity condition, we need to satisfy
  \[\varphi_e(-c_et) = \varphi_{\tilde{e}}(-c_{\tilde{e}}t) \quad(e,\tilde{e}\in E, t\in\R).\]
  Put differently, there needs to be $\varphi\in C^3(\R)$ such that 
  \[\varphi_e(t) = \varphi(k_e t) \quad(e\in E, t\in\R)\]
  for some $k_e\in \R$. In particular $k_e c_e = k_{\tilde{e}} c_{\tilde{e}}$ for all $e,\tilde{e}\in E$.
  Since by Theorem \ref{thm:KdV_R_solitary_wave} we know how $\varphi_e$ looks like, we thus need
  \[-\frac{3(\beta_e+c_e)}{\gamma_e}\sech^2\Bigl(\frac{\sqrt{\frac{\beta_e+c_e}{\alpha_e}}}{2}(-c_et-y_{0,e})\Bigr) = -\frac{3(\beta_{\tilde{e}}+c_{\tilde{e}})}{\gamma_{\tilde{e}}}\sech^2\Bigl(\frac{\sqrt{\frac{\beta_{\tilde{e}}+c_{\tilde{e}}}{\alpha_{\tilde{e}}}}}{2}(-c_{\tilde{e}}t -y_{0,{\tilde{e}}})\Bigr) \quad (e,\tilde{e}\in E, t\in\R).\]
  The second, third and forth assumption on the coefficients now yield that this is indeed true. Hence, the first coupling condition can be achieved.
  
  For the second coupling condition, we note that for $e, \tilde{e}\in E$ we have  
  \[u_{\tilde{e}}(t,x) = \varphi_{\tilde{e}}(x-c_{\tilde{e}}t) = \varphi_e\bigl(\frac{c_e}{c_{\tilde{e}}}x - c_e t\bigr) \quad(t\in\R, x\in\Omega_{\tilde{e}}).\]
  Hence,
  \[\partial_x^2 u_{\tilde{e}}(t,0) = \bigl(\frac{c_e}{c_{\tilde{e}}}\bigr)^2 \varphi_e''\bigl( - c_e t\bigr) = \frac{1}{c_{\tilde{e}}^2} c_e^2\varphi_e''(-c_et) \quad(t\in\R).\]
  Therefore, the second coupling condition requires
  \[\sum_{\tilde{e}\in E_-} -\alpha_{\tilde{e}} \frac{1}{c_{\tilde{e}}^2} c_e^2\varphi_e''(-c_et) + \sum_{\tilde{e}\in E_+} \alpha_e \frac{1}{c_{\tilde{e}}^2} c_e^2\varphi_e''(-c_et) + \Bigl(\sum_{e\in E_-} \frac{\beta_e}{2}  \sum_{e\in E_+} -\frac{\beta_e}{2}\Bigr) \varphi_e(-c_e t) = 0,\]
  or put differently
  \[\Bigl(\sum_{\tilde{e}\in E_-} -\alpha_{\tilde{e}} \frac{1}{c_{\tilde{e}}^2} + \sum_{\tilde{e}\in E_+} \alpha_{\tilde{e}} \frac{1}{c_{\tilde{e}}^2}\Bigr) c_e^2\varphi_e''(-c_et) + \Bigl(\sum_{e\in E_-} \frac{\beta_e}{2} + \sum_{e\in E_+} -\frac{\beta_e}{2}\Bigr) \varphi_e(-c_e t) = 0 \quad(t\in\R).\]
  Now, the fifth and sixth assumption on the coefficients yield the assertion.
  
  Let us turn to the third coupling condition. We have
  \[\partial_x u_e(t,0{ {\scriptstyle{\pm}}}) = \varphi_e(-c_e t) = k_e \varphi'(-k_e c_e t) \quad(t\in\R, e\in E_\pm).\]
  Hence, there exists $A\in C^2(\R)$ such that
  \[\partial_x u_e(t,0{ {\scriptstyle{\pm}}}) = A(t) \frac{1}{c_e} \quad(e\in E_\pm).\]
  By assumption on $U$ we have
  \[U\partial_x u(t,0\rlim) = A(t) U \bigl(\frac{1}{c_e}\bigr)_{e\in E_+} = A(t) \bigl(\frac{1}{c_e}\bigr)_{e\in E_-} = \partial_x u(t,0\llim). \qedhere\]  
\end{proof}

As a direct consequence we obtain that if the parameter set on any edge is related by scaling to the parameter set of any other edge, then solitary waves require equality of the parameter sets over the edges.

\begin{corollary}
	For coupling conditions as in the previous theorem, if for all $e_1,e_2\in E$ there exists $p_{21}\in\R$ such that
	\[(\alpha_{e_2}, \beta_{e_2}, \gamma_{e_2}, c_{e_2}) = p_{21}(\alpha_{e_1}, \beta_{e_1},\gamma_{e_1}, c_{e_1})\]
	and if there exists a solitary wave, then 
	\[(\alpha_{e_2}, \beta_{e_2}, \gamma_{e_2}, c_{e_2}) = (\alpha_{e_1}, \beta_{e_1},\gamma_{e_1}, c_{e_1})\]
	for all $e_1,e_2\in E$.
\end{corollary}

We now fix $(\alpha_e)_{e\in E}$ and $(\beta_e)_{e\in E}$ (the parameters of the linearized part) to be constant, so equal on all edges. Then existence of solitary waves requires the remaining parameters to be constant and the metric graph to be \emph{balanced}, i.e.\ $|E_-| = |E_+|$.

\begin{corollary}
	Let $Y:=\lin(\1) \subseteq \K^{E_-}\oplus \K^{E_+}$, and $U\from(\K^{E_+},(\alpha_e)_{e\in E_+}) \to (\K^{E_-},(\alpha_e)_{e\in E_+-})$ contractive. Let $(\alpha_e)_{e\in E}$ and $(\beta_e)_{e\in E}$ be constant.
	Consider \eqref{eq:KdV_star_graph} with coupling conditions
	\begin{align*}
	\bigl(u(t,0\llim),u(t,0\rlim)\bigr) & \in Y, \\
	\begin{pmatrix} -\alpha_- & 0 \\ 0 & \alpha_+ \end{pmatrix}\begin{pmatrix} \partial_x^2 u(t,0\llim) \\ \partial_x^2 u(t,0\rlim)\end{pmatrix} + \begin{pmatrix} \frac{\beta_-}{2} & 0 \\ 0 & -\frac{\beta_+}{2} \end{pmatrix}\begin{pmatrix} u(t,0\llim) \\ u(t,0\rlim)\end{pmatrix} & \in Y^\bot,\\
	\partial_x u(t,0\llim) & = U \partial_x u(t,0\rlim)
	\end{align*}
	for all $t\in\R$.
	Then \eqref{eq:KdV_star_graph} admits a solitary wave with $c_e\geq -\frac{2}{3}\beta_e$ for all $e\in E$ if and only if and
	$U\bigl(\frac{1}{c_e}\bigr)_{e\in E_-} = \bigl(\frac{1}{c_e}\bigr)_{e\in E_+}$,  $\abs{E_-} = \abs{E_+}$, and $(\gamma_e)_{e\in E}$, $(c_e)_{e\in E}$ and $(y_{0,e})_{e\in E}$ are constant. 
\end{corollary}

\begin{proof}
	Assume we have solitary waves with $c_e\geq -\frac{2}{3}\beta_e$ for all $e\in E$.	
	The equality $\sum_{e\in E_-} \beta_e + \sum_{e\in E_+} \beta_e = 0$ and constancy of $(\beta_e)_{e\in E}$ yields $\abs{E_-} = \abs{E_+}$.
	The equality $\sqrt{\frac{\beta_e+c_e}{\alpha_e}} c_e = \sqrt{\frac{\beta_{\tilde{e}}+c_{\tilde{e}}}{\alpha_{\tilde{e}}}} c_{\tilde{e}}$ for $e,\tilde{e}\in E$ and constancy of $(\alpha_e)_{e\in E}$ and $(\beta_e)_{e\in E}$
	yields constancy of $(c_e)_{e\in E}$ since $[-\frac{2}{3}\beta ,\infty)\ni c\mapsto \sqrt{\beta+c} c$ is injective.
	Then $y_{0,e}\frac{1}{c_e} = y_{0,\tilde{e}}\frac{1}{c_{\tilde{e}}}$ for $e,\tilde{e}\in E$ yields constancy of $(y_{0,e})_{e\in E}$, and $
	\frac{(\beta_e+c_e)}{\gamma_e} = \frac{(\beta_{\tilde{e}}+c_{\tilde{e}})}{\gamma_{\tilde{e}}}$ for $e,\tilde{e}\in E$ yields constancy of $(\gamma_e)_{e\in E}$.
	
	To show the converse, if the graph is balanced and $(\gamma_e)_{e\in E}$, $(c_e)_{e\in E}$ and $(y_{0,e})_{e\in E}$ are constant, we easily see from Theorem \ref{thm:KdV_R_solitary_wave} that we then obtain a solitary wave.
\end{proof}

The previous corollary actually states that we can view the graph as a bunch of $\abs{E_-}$ real lines and the solitary wave on the graph is nothing but the same solitary wave on each of this real lines.

\bigskip

If we drop the continuity assumption at the vertex $0$, we obtain somewhat non-trivial solitary waves. The proof is analogous to the one of Theorem \ref{thm:KdV_R_solitary_wave}. We demonstrate this by means of a so-called \emph{$Y$-junction}, i.e.\ a metric star graph with three edges.

\begin{proposition}
	Let $\abs{E_-}=1$, $\abs{E_+}=2$, $a\in\R\setminus\{0\}$, $Y:=\lin\{(1,a,a)\}$, and $U\from(\K^{E_+},(\alpha_e)_{e\in E_+}) \to (\K^{E_-},(\alpha_e)_{e\in E_-})$ contractive. Then \eqref{eq:KdV_star_graph} 
	with coupling conditions
  \begin{align*}
    \bigl(u_-(t,0\llim),u_{+,1}(t,0\rlim),u_{+,2}(t,0\rlim)\bigr) & \in Y, \\
    \begin{pmatrix} -\alpha_- & 0 & 0\\ 0 & \alpha_{+,1} & 0 \\
    0 & 0 & \alpha_{+,2}\end{pmatrix}\begin{pmatrix} \partial_x^2 u_-(t,0\llim) \\ \partial_x^2 u_{+,1}(t,0\rlim)\\\partial_x^2 u_{+,2}(t,0\rlim)\end{pmatrix} + \begin{pmatrix} \frac{\beta_-}{2} & 0 & 0 \\ 0 & -\frac{\beta_{+,1}}{2} & 0 \\ 0 & 0 & -\frac{\beta_{+,2}}{2} \end{pmatrix}\begin{pmatrix} u_-(t,0\llim) \\ u_{+,1}(t,0\rlim)\\ u_{+,2}(t,0\rlim)\end{pmatrix} & \in Y^\bot,\\
    \partial_x u_-(t,0\llim) & = U \begin{pmatrix}\partial_x u_{+,1}(t,0\rlim)\\\\ \partial_x u_{+,2}(t,0\rlim)\end{pmatrix}
  \end{align*}
  for all $t\in\R$	
	admits a solitary wave if and only if
	\[U\begin{pmatrix} 1/c_{+,1}\\1/c_{+2}\end{pmatrix} = \frac{1}{c_{-}}\]
	and for all $j\in\{1,2\}$ we have
	\begin{align*}
	\beta_-+c_-, \; \beta_{+,j}+c_{+,j} & > 0,\\
	\frac{y_{0,c_-}}{c_{-}} & = \frac{y_{0,c_{+,j}}}{c_{+,j}},\\
	\sqrt{\frac{\beta_-+c_-}{\alpha_-}} c_- & = \sqrt{\frac{\beta_{+,j}+c_{+,j}}{\alpha_{+,j}}} c_{+,j},\\
	a\frac{(\beta_-+c_-)}{\gamma_-} & = \frac{(\beta_{+,j}+c_{+,j})}{\gamma_{+,j}},\\
	 \frac{\alpha_{-}}{c_{-}^2} & = a^2\bigl(\frac{\alpha_{+,1} }{c_{+,1}^2}+ \frac{\alpha_{+,2}}{c_{+,2}^2}\Bigr),\\
	\beta_-  & = a^2\bigl(\beta_{+,1} + \beta_{+,2}\Bigr).
	\end{align*}
\end{proposition}

\begin{proof}
    The proof is analogous to the proof of Theorem \ref{thm:KdV_R_solitary_wave}.
\end{proof}

%

\section*{Acknowledgements}

C.S.\ is grateful to Mitja R\"oder for useful discussions on the topic.

\bibliographystyle{alpha}
\bibliography{bibfile}{}

%
%
%

\end{document}